\documentclass[12pt,oneside]{amsart} 

\usepackage{amssymb,amscd}
\usepackage{euscript}
\usepackage{hyperref}
\usepackage{color}

     \theoremstyle{plain}
      \newtheorem{theorem}{Theorem}[section]
      \newtheorem{lemma}[theorem]{Lemma}
      \newtheorem{corollary}[theorem]{Corollary}
      \newtheorem{proposition}[theorem]{Proposition}
      \newtheorem{remark}[theorem]{Remark}
      
      \newtheorem{definition}[theorem]{Definition}      
       
      \newtheorem{example}[theorem]{Example}

       \numberwithin{equation}{section}

      \makeatletter
      \def\@setcopyright{}
      \def\serieslogo@{}
      \makeatother

\def\M{X}

\def\N{\mathbb{N}}
\def\R{\mathbb R}
\def\rd{{\mathbb R ^d}}
\def\Z{\mathbb Z}
\def\T{\mathbb T}  
\def\Q{\mathbb Q}  

\def\tE{\tilde {E}}
\def\E{\mathcal{E}}
\def\TE{\tilde {\mathcal{E}}}
\def\w{\mathcal{W}}

\def\H{\mathcal H}
\def\m{\mathcal{M}}

\def\A{\EuScript{A}} 
\def\B{\EuScript{B}}
\def\tA{\tilde {\EuScript{A}}} 
\def\tB{\tilde {B}}

\def\tC{\tilde {C}}

\def\dist{\text{dist}}

\def\Id{\text{Id}}
\def\e{\varepsilon}
\def\la{\lambda}

\def\V{{\mathcal{V}}}

\def \a{\alpha}

\def\b{\beta}

\def\h{\varphi}
\def\h{\varphi}

\def\h{{\varphi}}

\def\H{{\mathcal {H}}}

\def\t{\mathcal T}

\def\QED{\hfill\hfill{\square}}

\begin{document}

\date{\today}
\author{Boris Kalinin$^1$ and Victoria Sadovskaya$^2$}

\address{Department of Mathematics, The Pennsylvania State University, University Park, PA 16802, USA.}
\email{kalinin@psu.edu, sadovskaya@psu.edu}

\title [Periodic data rigidity for cocycles and hyperbolic automorphisms]
{Periodic data rigidity for cocycles \\and hyperbolic automorphisms} 


\thanks{$^1$  Supported in part by Simons Foundation grant 855238}
\thanks{$^2$ Supported in part by  Simons Foundation grant MP-TSM-00002874}


\begin{abstract}

We study cohomology of H\"older continuous linear cocycles over a hyperbolic dynamical system
 and regularity of conjugacy between Anosov systems.
  
 For cocycles $\A$ and $\B$ with conjugate periodic data, we establish H\"older cohomology under various conditions: the periodic data of $\B$ has narrow spectrum and the periodic data conjugacy $C(p)$ is H\"older continuous at a periodic point;  $\B$ is constant and the cocycles are measurably cohomologous; $\B$ is constant and diagonalizable over $\mathbb C$ and either its Lyapunov spaces are at most two-dimensional or $C(p)$ is in a bounded set. 
 
 We also prove that a topological conjugacy between a weakly irreducible hyperbolic automorphism $L$ and an Anosov diffeomorphism $f$ of $\T^d$ is smooth if  their derivative cocycles $L$ and $Df$ are conjugate. 
  Using this and our results on cohomology of cocycles  we obtain global periodic data rigidity results for weakly irreducible hyperbolic automorphisms. In the argument we also establish differentiability of stable holonomies in low regularity setting.
  
\end{abstract}

\maketitle 



\section{Introduction and results}

\subsection{Introduction} 
Cohomology of H\"older continuous cocycles over hyperbolic dynamical systems has been extensively studied. One of the main problems in this area is establishing cohomology based 
on the periodic data of the cocycles.
It was solved by Liv\v{s}ic in \cite{Liv1,Liv2} for scalar cocycles and, more
generally, for  cocycles with values in abelian groups. In the non-commutative setting, matrix-valued and linear cocycles are the primary objects of study, motivated by derivative cocycles of smooth systems. 
Cocycles play an important role in smooth dynamics and rigidity of Anosov systems. We formulate our results for linear cocycles in Section \ref{cocycles sec} and give applications to rigidity of hyperbolic automorphisms in Section \ref{rigidity}.
\vskip.05cm

Let $f: \M \to \M$ be a homeomorphism of a compact metric space and let $P : \E \to \M $ be a finite dimensional H\"older continuous vector bundle over $\M$. A {\em linear cocycle}  $\A$ over $f$ 
is an  automorphism of $\E$ that projects to $f$, that is, a homeomorphism of $\E$ such that
$$P \circ \A = f \circ P \quad\text{and $\;\A_x : \E_x \to \E_{fx}$\,
is a linear isomorphism for each $x\in \M$.}
$$
If $f$ is a diffeomorphism of a manifold $X$, then the differential $Df$ is a linear cocycle on the tangent bundle $T\M$.  
Another important class of examples is given by random and Markovian sequences of matrices. They correspond to locally constant cocycles on a trivial bundle over a full shift or a subshift of finite type.

In the case of a trivial bundle $\E=\M \times \rd$, the cocycle is defined by a function 
$\A: \M \to GL(d,\R)$ with $\A(x)=\A_x$, and it is called a $GL(d,\R)$  cocycle.
We use the following metric on $GL(d,\R)$
$$
d (A, B) = \| A  - B \|  + \| A^{-1}  - B^{-1} \|, \quad\text{where $\|\,.\,\|$ is the operator norm.} 
$$

A cocycle $\A$  is called $\beta$-H\"older if  $\A_x$ depends $\beta$-H\"older continuously on $x$. A detailed description of this notion in the bundle setting is given in Section 2.2 of \cite{KS13}.
\vskip.05cm 

Two $GL(d,\R)$ cocycles $\A$ and $\B$ over $f$ are {\it cohomologous} if there exists a function $C:\M\to GL(d,\R)$ such that
$$
  \A_x=C(fx) \circ \B_x \circ C(x)^{-1} \quad\text{for all }x\in \M.
$$
Such a function $C$ is called a  {\em conjugacy} or {\em transfer map} between $\A$ and $\B$.
For linear cocycles $\A:\E \to \E$ and $\B:E\to E$ over $f$, a conjugacy is defined similarly 
with $C(x) \in GL(E_x,\E_x)$. In this case $C:E\to \E$ is a bundle isomorphism.


\subsection{Continuous conjugacy from the periodic data} \label{cocycles sec}$\;$ 
\vskip.1cm
\noindent{\bf Assumptions.} In this section, $f:X \to X$ is a topologically transitive Anosov diffeomorphism, or a topologically mixing diffeomorphisms of a locally maximal hyperbolic set, or a mixing subshift of finite type (see Section \ref{base} for details).  All cocycles over $f$ are assumed to be H\"older continuous.
\vskip.1cm

We denote  the $n$th iterate of $\A$ by $\A^n$, so that   $\,\A^0=\Id\,$ and for  $x\in X$  and $n\in \N$,
$$
\A_x^n = \A_{f^{n-1} x}\circ \cdots \circ \A_x \;\text{ and }\;
\A_x^{-n}= (\A_{f^{-n} x}^n)^{-1}.
$$

We say that $GL(d,\R)$ cocycles $\A$ and $\B$ over $f$ have {\it conjugate periodic data} if for each periodic point $p=f^n p$ there is a matrix $C(p)\in GL(d,\R)$
 such that 
 $$ \A_p ^{n}=C(p)\circ \B_p ^{n} \circ C(p)^{-1}.$$
  For linear cocycles this means existence of such an operator $C(p) \in GL(E_p,\E_p)$.

 Clearly, continuous conjugacy between $\A$ and $\B$ implies conjugacy of their periodic data, and a natural question is whether the converse is true. This problem is difficult and far from fully solved. The following example shows that the answer may be negative  even in dimension two with a constant cocycle $\B$ and a bounded set $\{ C(p)\}$.  

\begin{example}\cite{S13} \label{periodic conj} 
There  exist a constant cocycle $\B_x=$ {\footnotesize $\left[ \begin{array}{cc} 1 &  \beta \\ 0& 1 \end{array} \right] \,$} and a smooth cocycle 
$\A_x=$ {\footnotesize $\left[ \begin{array}{cc} 1 &  \a(x) \\ 0& 1 \end{array} \right]$\,}  
with $\A_x$ and $B_x$
arbitrarily  close to the identity, such that  $\A$ and $\B$ have  conjugate periodic data with $\{C(p)\}$ bounded in $GL(2,\R)$, but are not measurably cohomologous. 
\end{example}

The case of $ \A_p ^{n}= \B_p ^{n}$ is simpler and  better understood. Positive answers were given
by Parry and Pollicott  \cite{PaP,Pa} for cocycles with vales in compact groups and by Schmidt \cite{Sch}  for cocycles  with ``bounded distortion".
For general fiber bunched linear cocycles with equal periodic data the problem was solved in \cite{S15,Ba}. 
Positive results for conjugate periodic data 
were obtained in \cite{Pa,Sch} for cohomology in compact
groups under extra assumption of transitivity of the cocycle extension on the skew-product. 

In light of Example \ref{periodic conj}, for linear cocycles one needs some regularity assumption on $C(p)$. 
We say that $C(p)$ is $\beta$-H\"older continuous at a periodic point $q$ if
\begin{equation}\label{Holder at q}
d(C(p),C(q))\le c\, \dist(p,q)^\beta\; \text{ for every periodic point $p$ close to $q$.}
\end{equation}
Under this assumption, H\"older cohomology between linear cocycles $\A$ and $\B$, where $\B$  is fiber-bunched (see Definition \ref{bunching def}), was obtained in \cite{S15,S17}. Further, H\"older cohomology was established for a  constant  $GL(d,\R)$-valued cocycle $\B$ and  a cocycle $\A$ sufficiently close to $\B$ with periodic data conjugacy satisfying \eqref{Holder at q}.
The following theorem extends this result to the global setting, and to $\B$ with $\delta$-narrow periodic data, as defined in \cite{DG24}.

\begin{definition}\label{delta narrow}
A cocycle $\B$ has {\em $\delta$-narrow periodic data} centered at $\la_1  \ge \la_2\ge \dots \ge  \la_d$ 
if for each periodic point $p=f^np$ the eigenvalues $\a_i$ of $\B_p^n$ can be ordered so that
\vskip.1cm \hskip3cm 
$e^{n(\la_i-\delta)} \le |\a_i| \le e^{n(\la_i+\delta)} \quad \text{ for each $i=1, \dots, d$}.$
\end{definition}

Clearly, this property holds for a constant cocycle with $\delta=0$, and for a cocycle close to constant with  a small  $\delta>0$.

\begin{theorem}   \label{Holder C(p)}
Let $\A$ and $\B$ be H\"older continuous linear cocycles over $f$ on bundles $\E$ and $E$. 
Suppose that $\A$ and $\B$ have conjugate periodic data such that  $C(p)$ can be chosen 
H\"older continuous at some periodic point $q$. Suppose also that $\B$ has
$\delta$-narrow periodic data with a sufficiently small $\delta$.
Then there exists a unique H\"older continuous conjugacy $\bar C$ between $\A$ and $\B$ with $\bar C(q)=C(q)$.
\end{theorem}

\begin{remark}The H\"older exponent of $\bar C$ may be smaller than the exponents  of $\A$ and $\B$, and it is discussed in the proof. The same holds in Theorems \ref{+meas} and \ref{cocycle theorem} below.
\end{remark}

 We note that  $\bar C(p)$ does  not necessarily coincide with 
$C(p)$ for $p\ne q$. For example, let $\B_x=\Id$ and
let $\A_x=\bar C(fx) \circ \bar C(x)^{-1}$,
where $\bar C:\M\to GL(d,\R)$ is any H\"older continuous function
with $\bar C (q)=\Id$.
Then $\A^n_p=\B^n_p=\Id\,$ whenever $p=f^np$, and thus we can take
$C(p)=\Id$ for each $p$.

\vskip.1cm

In the next theorem for a constant cocycle $\B$ we replace H\"older continuity of $C(p)$ at $q$ by existence of a measurable conjugacy between the cocycles. We note that there are results on continuity of a measurable conjugacy, but  not  in this setting.

\begin{theorem}   \label{+meas}
Let $\B=B$ be a constant $GL(d,\R)$ cocycle and let $\A$ be a H\"older  linear cocycle over $f$ with conjugate periodic data. Suppose there exists a  conjugacy $C$ between $\A$ and $\B$ 
which is measurable with respect to an ergodic $f$-invariant measure on $X$ with full support and local product structure. Then $C$ coincides on a set of full measure with  a H\"older continuous conjugacy between $\A$ and $\B$.
\end{theorem}

In Theorem \ref{+meas}  and Theorem \ref{cocycle theorem} below the linear cocycle $\A$ is on a vector bundle $\E$, and conjugacy of the cocycles yields that $\E$ is trivial.

In the next theorem  we consider a constant cocycle $\B$  diagonalizable over $\mathbb C$. Compared to Theorem \ref{Holder C(p)}, we remove the continuity  assumption on $C(p)$ if $\B$ has at most two-dimensional Lyapunov spaces and weaken it to boundedness in the general case. 

\begin{theorem}   \label{cocycle theorem}
Let $\B=B$ be a diagonalizable over $\mathbb C$ constant $GL(d,\R)$ cocycle over $f$ and let $\A$ be a H\"older linear cocycle over $f$ with conjugate periodic data.
Then $\A$ is H\"older conjugate to $\B$ if either of the following conditions holds
\vskip.05cm
\begin{itemize}
\item[{\bf (i)}] Lyapunov spaces of $B$ are at most two-dimensional, that is,  no three  eigenvalues of $B$ (counted with multiplicity)  have the same modulus; 
\vskip.1cm
\item[{\bf (ii)}]  The periodic data conjugacy $C(p)$ can be chosen in a bounded subset of $GL(d,\R)$. 
\end{itemize}
\vskip.1cm
\noindent 
\end{theorem}

This result is optimal for cohomology to a constant cocycle. 
Example \ref{periodic conj} shows that  the theorem does not hold without the diagonalizability assumption.
Also, in dimension at least 3 the boundedness assumption 
on $C(p)$ cannot be dropped, as the following example by de la Llave  in the appendix to \cite{GKS11} shows. 

\begin{example}\label{3d example}
Let $f$ be an Anosov diffeomorphism  of a manifold $\M$. There
exists a family of $SL(3, \R)$-valued cocycles
$\A_\epsilon$, $|\epsilon | < 1$, over $f$ such that:
\begin{itemize}
\item $\A_\epsilon(x)$ is jointly analytic in $\epsilon$ and $x$;

\item $\A_0=B$ is a  constant orthogonal matrix;

\item For any $\epsilon$ and any $p=f^np$, the matrix $\A_\epsilon^n (p)$ 
conjugate to $\A_0^n (p)=B^n$; 
\item For any $\epsilon \ne 0$, the cocycle $\A_\epsilon$ is not uniformly quasi-conformal,\\
and so  $\A_\epsilon$ cannot be continuously 
conjugate to $\A_0$.
\end{itemize}
\end{example}

In this paper we use a recent result by DeWitt and Gogolev on existence of dominated splitting
from periodic data \cite{DG24}. 
Combing it with results and techniques developed in \cite{S15,S17,KSW23}
we obtain Theorems \ref{Holder C(p)} and \ref{+meas}. Theorem \ref{cocycle theorem} 
uses a new approach  for removing/weakening continuity assumption on $C(p)$.
The argument draws on results and ideas from \cite{KS10, KS13,S15}.

\vskip.3cm


\subsection{Applications to rigidity problems for hyperbolic systems} \label{rigidity}$\;$
\vskip.1cm
\noindent 
We recall that a toral automorphism $L$ is {\it hyperbolic} if it has no eigenvalues of modulus 1. 
By the classical results of Franks and Manning \cite{F69,M73}, any Anosov diffeomorphism $f$ of $\T^d$ is topologically conjugate to the hyperbolic automorphism $L$ that $f$ induces on $\Z^d=H_1(\T^d,\Z)$. 
 A {\em topological conjugacy} is
 a ho\-meo\-morphism $h$ of $\T^d$ such that 
\begin{equation} \label{Conj dif}
L\circ h= h \circ f.
\end{equation}
Any two such conjugacies differ by an affine automorphism of $\T^d$ commuting with $L$ \cite{W},
and hence have the same regularity.  
A conjugacy $h$ in \eqref{Conj dif} is always bi-H\"older, but it is usually not even $C^1$, as there are various obstructions to smoothness. 

The question when $h$ is smooth has been extensively studied and periodic data played an important role \cite{L0,LM,L1,G08,GKS11,S15,DG24}. The problem is closely related to rigidity of cocycles.
Indeed, if $h$ is $C^1$ then differentiating   \eqref{Conj dif} we obtain 
\begin{equation} \label{Conj Dh}
L\circ Dh(x)= Dh (fx) \circ Df(x)
\end{equation}
and hence $C(x)=Dh(x)$ is a continuous conjugacy of the derivative cocycles $Df$ and $DL=L$.
Thus conjugacy of the derivative cocycles is necessary for smoothness of $h$. 

This condition is not sufficient in general for {\it reducible} $L$. 
This can be seen  in the example by de la Llave \cite{L1}  of an automorphism $L(x,y)=(Ax, By)$  of $\T^4$ where $A,B\in SL(2,\R)$
have eigenvalues $\lambda, \lambda^{-1}$ and 
 $\mu, \mu^{-1}$, respectively, with $\mu> \lambda>1$. 
A perturbation  
$$f(x,y)= (Ax+ \varepsilon \sin (2\pi y_1)v,\, By),
$$
 where $v$ is an eigenvector of $A$ corresponding to $\lambda$,  is not $C^1$ conjugate to $L$. For any $p=f^np$, the matrix $Df^n_p$ has eigenvalues
$\lambda^n, \lambda^{-n}, \mu^n,$ $\mu^{-n}$  and hence is conjugate to $L^n$, and Theorem \ref{cocycle theorem}(i) yields that cocycles $Df$ and $L$ are H\"older cohomologous. 
\vskip.1cm

Our main result in this section, Theorem \ref{relation theorem}, shows that  conjugacy of the derivative cocycles is sufficient for {\it weakly irreducible} $L$. This assumption is weaker than irreducibility.
We recall that $L\in SL(d,\Z)$ is {\it irreducible} if its characteristic polynomial is irreducible over $\Q$, equivalently,  if it has no nontrivial rational invariant subspaces. An automorphism $L$ is irreducible if and only if every $L$-invariant linear foliation of $\T^d$ has dense leaves. The eigenvalues of an irreducible $L$ are simple, and hence $L$ is diagonalizable over $\mathbb C$, though different eigenvalues may have the same modulus. 
 The notion of weak irreducibility was introduced in \cite[Section 3.3]{KSW23} in the context of bootstrapping the regularity of $h$ to $C^\infty$, and further discussed in \cite[Section 2.2]{KS25}.  
 
 \begin{definition} We say that $L\in SL(d,\Z)$ is {\em weakly irreducible} if any of the following equivalent conditions holds:
\begin{itemize}
\vskip.1cm
\item[{\bf (i)}]  The leaves of each Lyapunov foliation of $L$ are dense in $\T^d$,
\vskip.1cm
\item[{\bf (ii)}]   All irreducible over $\Q$ factors of the characteristic polynomial of $L$ have the same set of moduli of the roots.
\end{itemize}
\end{definition}

A weakly irreducible $L$ is not necessarily irreducible and not necessarily  diagonalizable. For example, if $A$ is weakly irreducible, then so are 
$$
L= \small\left(\begin{array}{cc} A & 0\\ 0 & A \end{array}\right) \quad \normalsize \text{and}\quad L=\small\left(\begin{array}{cc} A & I\\ 0 & A \end{array}\right).
$$
Moreover, if $A,B\in SL(d,\Z)$ are weakly irreducible with the same set of moduli of eigenvalues, then 
$$
L=\small\left(\begin{array}{cc} A & 0\\ 0 & B \end{array}\right)\quad \normalsize \text{is weakly irreducible.}
$$

\begin{theorem}   \label{relation theorem}
Let $L:\T^d\to\T^d$ be a weakly irreducible  hyperbolic automorphism, and let  $f$ be a $C^{1+\text{H\"older}}$ diffeomorphism of $\T^d$ topologically conjugate to $L$. \\Assume one of the following
\begin{itemize}
\item[{\bf (i)}]   The derivative cocycle $Df$ is continuously conjugate to $L$,
\vskip.1cm
\item[{\bf (ii)}]  The diffeomorphism $f$ is sufficiently $C^1$ close to $L$, and there is a conjugacy between $Df$ and $L$ measurable 
with respect to an ergodic invariant measure (for $f$ or for $L$)  with full support and local product structure.
\end{itemize}
Then $f$ is $C^{1+\text{H\"older}}$ conjugate to $L$. Further, if $f$ is $C^\infty$ then it is $C^\infty$
conjugate to $L$.
\end{theorem}

The $C^\infty$ regularity is obtained using \cite{KSW25}.
The weak irreducibility assumption is dynamically natural, and examples of Gogolev in \cite{G08} show that this assumption is necessary in the theorem. The proof of this theorem shares a general approach
with \cite{GKS11}, however the proofs of the two main steps there, corresponding to Proposition \ref{smooth along W} and verifying its assumption here, rely on conformality of $L$ on its Lyapunov subspaces.
In our setting possibility of Jordan blocks requires different arguments, and the proof of Proposition \ref{smooth along W} uses a completely different approach. The only prior result for automorphisms
with Jordan blocks was obtained by DeWitt \cite{D} in dimension 4. 

In the proof of Theorem \ref{relation theorem} we also establish a result of independent interest, Lemma \ref{smooth hol},
on $C^{1+\text{H\"older}}$ regularity of stable holonomies in $C^{1+\text{H\"older}}$ setting.

\vskip.1cm 
Combining Theorem \ref{relation theorem}(i) with our results for cocycles  we obtain the following. 

\begin{corollary}   \label{rigidity corollary}
Let $L:\T^d\to\T^d$ be a weakly irreducible  hyperbolic automorphism. Let  $f$ be a $C^{1+\text{H\"older}}$ Anosov diffeomorphism of $\T^d$ topologically conjugate to $L$  with conjugate periodic data, that is, 
for each periodic point $p=f^n p$ there is
$$
C(p)\in GL(d,\R)\;\text{ such that }\; D_p f^{n}=C(p)\circ L^n \circ C(p)^{-1}.
$$
 Suppose that one of the following holds
\vskip.05cm
\begin{itemize}
\item[{\bf (i)}]  $L$ is diagonalizable over $\mathbb C$\, and  no three of  its eigenvalues have the same modulus, 
\vskip.1cm
\item[{\bf (ii)}]  $L$ is diagonalizable over $\mathbb C$\, and $C(p)$ is in a compact subset of $GL(d,\R)$,
\vskip.1cm
\item[{\bf (iii)}]  $C(p)$ is H\"older continuous at a periodic point of $f$, as in \eqref{Holder at q}.
\end{itemize}
\vskip.1cm
\noindent Then $f$ is $C^{1+\text{H\"older}}$ conjugate to $L$. 
Further, if $f$ is $C^\infty$ then any conjugacy is $C^\infty$.
\end{corollary}

H\"older conjugacy of the derivative cocycles $Df$ and $L=DL$ is given  by Theorem \ref{cocycle theorem} in (i) and (ii), and by Theorem \ref{Holder C(p)} in (iii). Then 
Theorem  \ref{relation theorem}(i) yields smooth conjugacy of $f$ and $L$.

Corollary \ref{rigidity corollary}  is the most general periodic data rigidity result for Anosov toral automorphisms. 
Part (i) extends the similar global rigidity result in \cite{DG24} for  irreducible $L$ to the weakly irreducible setting. Part (ii) gives an alternative to the assumption on the eigenvalues of $L$. 
Part (iii) generalizes a similar  result in  \cite{S15} for  perturbations of an irreducible $L$. H\"older assumption \eqref{Holder at q} yields smoothness of the conjugacy with minimal assumptions 
on $L$ and $f$.

The diagonalizability assumption in (i) and (ii) is necessary. In \cite{L2} de la Llave gave an example of an automorphism 
$$
L=\small{\left(\begin{array}{cc} A & I\\ 0 & A\end{array}\right)}, \quad \text{ where $A \in SL(2,\Z)$ is hyperbolic,}
$$
and  its  analytic perturbation $f$ with conjugate periodic data that is not $C^1$ conjugate to $L$.
Clearly, $L$ does not have three eigenvalues of the same modulus. One can also see that the $C(p)$ can be chosen bounded. For such automorphisms of $\T^4$, DeWitt \cite{D} identified an additional condition for the periodic data that ensures smooth  conjugacy between $f$ and $L$. 
\vskip.1cm

In Section \ref{base} we describe the three main classes  of hyperbolic systems. In Section \ref{cocycle proofs} we prove our results for cocycles, Theorems \ref{Holder C(p)}, \ref{+meas}, and  \ref{cocycle theorem}. In Section \ref{rigidity proof} we prove Theorem \ref{relation theorem}.


\section{Hyperbolic systems in the base} \label{base}

We consider cocycles over hyperbolic dynamical systems. 
Below we describe the three main classes  of such systems.
\vskip.2cm
\noindent{\bf Transitive Anosov diffeomorphisms.} 
A diffeomorphism  $f$ of a compact connected  manifold $X$
 is called {\it Anosov}\, if there exist a splitting 
of the tangent bundle $TX$ into a direct sum of two $Df$-invariant 
continuous subbundles $\E^s$ and $\E^u$,  a Riemannian 
metric on $X$, and  continuous  
functions $\nu$ and $\hat\nu$  such that 
\begin{equation}\label{Anosov def}
\|Df_x(v^s)\| < \nu(x) < 1 < \hat\nu(x) <\|Df_x(v^u)\|
\end{equation}
for any $x \in X$ and any unit vectors  
$v^s\in \E^s(x)$ and $v^u\in \E^u(x)$.
The sub-bundles $\E^s$ and $\E^u$ are called stable and unstable. 
They are tangent to the stable and unstable foliations 
$\w^s$ and $\w^u$ respectively. 
The {\it local stable manifold} of $x$, $\,\w^s_{\text{loc}}(x)$, is a ball 
 centered at $x$ of radius $\rho$ in the intrinsic metric of $\w^{s}(x)$. 
We choose $\rho$  sufficiently small   so that $\w^s_{\text{loc}}(x)\cap \w^u_{\text{loc}}(z)$ consists of a single point 
for any sufficiently close $x$ and $z$ in $X$.
Local unstable manifolds are defined similarly. 
\vskip.05cm

A diffeomorphism $f$ is {\it (topologically) transitive} if there is a point $x$ in $\M$
with dense orbit. All known examples of Anosov diffeomorphisms have this property.
Any transitive Anosov diffeomorphism is {\it topological mixing}, that is, for any open sets $U_1,U_2 \subset X$ there exists $N\in\N$ such that $f^n(U_1)\cap U_2\ne \emptyset$ for all $n\ge N$.
\vskip.2cm


\noindent{\bf Topologically mixing diffeomorphisms of locally maximal hyperbolic sets.}  
More generally, let $f$ be a diffeomorphism of a manifold $\m$.
A compact $f$-invariant  set $X \subseteq \m$ is
called {\em hyperbolic} if there  exist a continuous $Df$-invariant splitting 
$T_X \m = E^s\oplus E^u$, and a Riemannian metric and 
continuous functions $\nu$, $\hat \nu$ on an open set
$U \supseteq X$ such that \eqref{Anosov def} holds for all $x \in X$. We assume that $f|X$ is topologically mixing.

The set $X$ is called {\em locally maximal} if 
$X= \bigcap_{n\in \Z} f^{-n }(U)$ for some open set $U\supseteq X$.


\vskip.2cm
\noindent{\bf Mixing subshifts of finite type.}
Let $M$ be $k \times k$ matrix with entries from $\{ 0,1 \} $ such that all 
entries of $M^N$ are positive for some $N\in \N$. Let
\vskip.15cm
$\hskip1cm X= \{ \,x=(x_n) _{n\in \Z}\, : \,\; 1\le x_n\le k \;\text{ and }\;
 M_{x_n,x_{n+1}}=1 \,\text{ for every } n\in \Z \,\}.$
\vskip.15cm
\noindent The shift map $f:X\to X\,$ is defined by 
$(fx)_n=x_{n+1}$.
The system $(X,f)$ is called a {\em  mixing  subshift of finite type}. 
We fix $\nu \in (0,1)$ and consider the metric 
\vskip.15cm
$\hskip1.5cm \dist(x,y) = d_\nu(x,y)=\nu^{n(x,y)},
\;\text{ where }\;n(x,y)=\min\,\{ \,|i|\,: \; x_i \ne y_i  \}.
$
\vskip.15cm
\noindent The following sets play the role of the local stable and unstable 
manifolds of $x$:
$$
\w^s_{\text{loc}}(x)=\{\,y\; | \;\, x_i=y_i, \;\;i\ge 0\,\}, \quad 
\w^u_{\text{loc}}(x)=\{\,y\; | \;\, x_i=y_i, \;\;i\le 0\,\},
$$
and we can take $\nu(x)=\nu$ and $\hat\nu (x) =\nu^{-1}$.

\section{Proofs of Theorems \ref{Holder C(p)}, \ref{+meas}, and  \ref{cocycle theorem}}\label{cocycle proofs}

\subsection{Existence of a dominated splitting.}$\;$\\
In the proofs, will be using the following result by DeWitt and Gogolev.

\begin{theorem}\cite[Theorems 1.3 and 3.10]{DG24} \label{DG}
Let $(X,f)$ be a transitive invertible subshift of finite type,  or more generally a homeomorphism of a compact metric space satisfying the Closing Property \cite[Definition 3.9]{DG24}.
Let $\la_1 \ge \la_2 \ge \dots \ge \la_d$ with $\la_k >\la_{k+1}$. For each $\beta' \in (0,1)$ there exists
$\delta>0$  such that if $\B$ is a $\beta'$-H\"older linear cocycle  with $\delta$-narrow periodic data centered at $(\la_1, \dots, \la_d)$ then $\B$ has a dominated splitting of index $k$.
\end{theorem}

We will apply this result to mixing hyperbolic dynamical systems described in Section \ref{base}, for which the Closing Property holds.  

A continuous splitting $E=G \oplus F$ is called a {\em dominated splitting} for a cocycle $\B$ on $E$ if 
it is $\B$-invariant and there exist constants $K>0$ and $0 < \tau< 1$ such 
that for every $x\in X$ and $n\in \N$, $$\| \B^n_x|F_x\| < K\tau^n m(\B^n_x|G_x),$$
where $m(B) =\| B^{-1}\|^{-1}$ denotes the conorm of a linear operator $B$. 
We say that the dominated splitting has index $k$ if $\dim G = k$.

\vskip.5cm

\subsection{Proof of Theorem \ref{Holder C(p)}.} $\;$\\
We assume that $\A$ and $\B$ are H\"older continuous linear cocycles with some exponent $\beta'>0$. 
We list the {\em distinct} values of $\la_1 , \dots, \la_d$ from the Definition \ref{delta narrow} of
$\delta$-narrow periodic data for  $\B$ as $\sigma_1 < \dots < \sigma_\ell$ and
denote by $d_1 , \dots, d_\ell$ their multiplicities.  If $\delta>0$ is sufficiently small, we can apply Theorem \ref{DG} to $\B$
and each gap $\sigma_i< \sigma _{i+1}$ and obtain the corresponding  dominated splitting 
for $\B$ of the bundle $E$
$$E =  E^1 \oplus \cdots \oplus   E^\ell$$
into the direct sum of continuous invariant subbundles with $\dim E^i=d_i$ for $i=1,\dots, \ell$.
Thus $\B$ splits into the direct sum of continuous cocycles $\B_i=\B| E^i$.

The subbundles $E^i$ and hence cocycles $\B_i$ are H\"older continuous with some exponent
$\hat \beta >0$. Specifically, there are $\bar \beta =\bar \beta (f, \sigma_1 , \dots , \sigma_\ell)$ and $\delta_0>0$ such that for all $\delta<\delta_0$ we can take
$\hat \beta = \min \{\beta ' , \bar \beta \}$.
Let $\beta''$ be the exponent of H\"older continuity of $C(p)$ at $q$. 
We set $\beta=\min \{\beta '', \hat \beta \}>0$. We will view $\B$ and $\B_i$ as  $\beta$-H\"older cocycles
and show the existence of $\beta$-H\"older conjugacy $\bar C$ between $\A$ and $\B$.

For each $\B_i$, the $\delta$-narrow assumption implies that the Lyapunov exponents of $\B_i$ 
at each periodic orbit are in the interval $[\sigma_i-\delta,\sigma_i+\delta]$. Then by 
\cite[Theorem 1.3]{K11} for any $\delta'>\delta$ there is constant $c$ such that
$$
 c^{-1}  e^{n(\sigma_i-\delta')} \|u\| \le \| \B_i^n u \|\le c\,e^{n(\sigma_i+\delta')}  \|u\|
\quad \text{for any $i$ and any vector }u\in E^i.
$$ 
In particular, if $\delta$ is small enough, then the $\beta$-H\"older cocycles $\B_i$ are $\beta$ fiber bunched.

\begin{definition} \label{bunching def}
 A $\beta$-H\"older  cocycle $\B$ over a hyperbolic system $(X,f)$ is\,
 {\em $\beta$ fiber bunched} if 
there exist numbers $\theta<1$ and $L$  such that for all $x\in X$ and $n\in \N$,
$$\| \B_x^n\| \cdot \| (\B_x^n)^{-1}\| \cdot  (\nu^n_x)^\beta < L\, \theta^n \quad\text{and}\quad
\| \B_x^{-n}\| \cdot \| (\B_x^{-n})^{-1}\| \cdot  (\hat \nu^{-n}_x)^\beta < L\, \theta^n, 
$$
\noindent where $\nu$ and $\hat \nu$ are as in \eqref{Anosov def}, 
$ \;\nu^n_x=\nu(f^{n-1}x)\cdots\nu(x)$  and  $\, \hat\nu^{-n}_x = (\hat\nu ^n_{f^{-n}x})^{-1}$.
\end{definition}

Since $\A$ and $\B$ have conjugate periodic data, $\A$ also has $\delta$-narrow periodic 
data centered at the exponents $\la_i$. Hence we also have  the corresponding 
 splittings for $\A$ 
  $$\E=\E^1 \oplus \dots \oplus \E^\ell \qquad\text{and}\qquad \A=\oplus \A_i,$$
where $\A_i=\A| \E^i$ satisfy the same estimates and are $\beta$-H\"older and fiber bunched.

At each periodic point the conjugacy $C(p)$ maps the splitting $E^1_x \oplus \cdots \oplus   E^\ell_x$
to the splitting $\E^1_x \oplus \cdots \oplus   \E^\ell_x$. Hence $C$ induces the conjugacy $C_i(p)$ of the periodic data
for $\B_i$ and $\A_i$, which is also $\beta$-H\"older continuous at $q$. 
Then  by Theorem 1.9 in \cite{S17}, for each $i$ there exists a unique $\beta$-H\"older conjugacy
$\bar C_i$ between cocycles $\B_i$ and $\A_i$ with $\bar C_i (q)=C_i(q)$. Then $\bar C=\oplus \bar C_i$ is the claimed $\beta$-H\"older conjugacy between $\A$ and $\B$.

\vskip.5cm
\subsection{Proof of Theorem \ref{+meas}.}
 As in the proof of Theorem \ref{Holder C(p)}, we obtain the splittings of $\B$ into the direct sum of constant cocycles  $\B_i=\B| E^i$ and of  $\A$
 into the direct sum of H\"older continuous cocycles $\A_i=\A| \E^i$.
Since  each  cocycle $\B_i$ is constant with one Lyapunov exponent $\la_i$, it is fiber bunched.
The conjugacy of the periodic data implies that 
 all Lyapunov exponents of $\A_i$ at the periodic orbits are also $\la_i$. 
 By periodic approximation of Lyapunov exponents \cite[Theorems 1.4]{K11},
 this  also holds for Lyapunov exponents of $\A_i$ for any $f$-invariant measure. Thus each cocycle $\A_i$ 
 has one Lyapunov exponent $\la_i$ with respect to every $f$-invariant  measure. 
 
Any measurable conjugacy preserves the Lyapunov exponents of vectors, see e.g.  \cite[Lemma 4.4] {KSW23}. Hence the conjugacy $C$ maps  $\mu$ a.e. the splitting $E^1_x \oplus \cdots \oplus   E^\ell_x$
to the splitting $\E^1_x \oplus \cdots \oplus   \E^\ell_x$, and so it induces  
 $\mu$-measurable conjugacy $C_i$ between $\B_i$ and $\A_i$. Since $\B_i$ is fiber bunched and $\A_i$ has 
 one Lyapunov exponent, $C_i$ coincides $\mu$ a.e. with a H\"older continuous conjugacy  
 by \cite[Theorem 2.1] {KSW23}. This holds for each $C_i$, and hence we conclude the same
 for their direct sum $C$.

\vskip.5cm

\subsection{Proof of Theorem \ref{cocycle theorem}.}
Let $\B_x =B$ be the constant generator of the cocycle $\B$, which is diagonalizable over $\mathbb C$. 
Let $\rho_1 < \dots <\rho_\ell$  be the distinct moduli of the eigenvalues of $B$ and let 
$\R^d = E^1 \oplus \cdots \oplus E^\ell$ be the Lyapunov splitting for $\B$, that is the invariant splitting 
into the direct sums of the eigenspaces corresponding to eigenvalues of modulus $\rho_i$, $i=1, \dots, \ell$. 
We denote $B_i=B| E^i$. Since  $\B$ is diagonalizable over $\mathbb C$, there exists an inner product on $\R^d$ with respect to which
$$
 \| B_i^n u \|=\rho_i^n \|u\|
\quad \text{for every $i=1, \dots, \ell$,\, every vector $u\in E^i$  and every $n$.}
$$ 
In particular, the  cocycle $\B_i$ generated by $B_i$ is fiber bunched for any $\beta>0$. 

As in the proof of Theorem \ref{Holder C(p)} we obtain the corresponding H\"older continuous 
invariant splitting $\E=\E^1 \oplus \dots \oplus \E^\ell$ for $\A$ and split $\A$ into the sum of H\"older continuous cocycles $\A_i=\A| \E^i$.
For any $p=f^np$, the restriction of the conjugacy $C(p)$ to $E^i$ conjugates $(\A_i)^n_p$ and $(\B_i)^n_p$. Since all $(\B_i)^n_p$ are conformal, it follows from our results in \cite{KS10}  that the cocycle $\A_i$ is also conformal with respect to some H\"older continuous  Riemannian metric on $\E^i$. Specifically, 
in 2-dimensional case (i) this follows from \cite[Theorem 1.3]{KS10} and in the general case
(ii) under the boundedness assumption on $C(p)$ this follows from \cite[Theorem 1.1]{KS10}. 

With respect to this metric, the norm $a_i(x)=\|\A_i(x)\|$ is an $\R_+$-valued cocycle with 
periodic data equal to that of the constant cocycle $\rho_i$. Hence by the  Liv\v{s}ic periodic point theorem, 
the cocycle $a_i$ is H\"older continuously 
cohomologous to the constant cocycle $\rho_i$ via a positive function $\phi(x)$. Then we can rescale the  Riemannian metric on $\E^i$ by $\phi(x)$ and obtain that $\|\A_i(x)\|=\rho_i$ with respect to the adjusted metric. Thus we also have
$$
 \| \A_i^n v \|=\rho_i^n \|v\|
\quad \text{for every vector $v\in \E^i$ and every $n$.}
$$ 
We will show that the cocycles $\A_i$ and $\B_i$ are 
cohomologous by some H\"older continuous conjugacy $C_i(x)$. Then $C=\oplus C_i$
conjugates $\A$ and $\B$.
\vskip.2cm

From now on we fix $i$.  If dim$(E^i)=1$, then H\"older conjugacy of $\A_i$ and $\B_i$ is given by the Liv\v{s}ic periodic point theorem. So from now on we assume that dim$(E^i)>1.$
We denote $\tA=(\rho_i)^{-1} \A_i$ and $\tB=(\rho_i)^{-1}B_i$.  
Then $\tB$ is a constant isometric cocycle on the trivial bundle $\tE=E^i$
and $\tA$ is a H\"older continuous cocycle on the H\"older continuous bundle $\TE=\E^i$
 isometric with respect to a H\"older continuous Riemannian metric $\sigma$ on $\TE$.
 The periodic data of these  cocycles are conjugate by $\tC(p)=C_i(p)=C|E^i(p)$.

  In case (ii) the set $\{\tC(p)\}$ is bounded by the assumption. We claim that 
 in case~(i) we also can choose $\{\tC(p)\}$ to be bounded. Indeed, we take any periodic point
 $p=f^n(p)$ and note that $\tA_p^n$ is isometric with respect to an inner product $\sigma (p)$ on $\TE$.
 Since $\sigma$ is continuous and thus bounded, we can choose  in a bounded way linear isomorphisms 
 $D_p: \TE_p \to \R^{2}$ which map $\sigma (p)$ to the standard inner product on $\R^{2}$. 
 Then $D_p\circ \tA_p^n \circ D_p^{-1}$ is an orthogonal matrix conjugate to the ortogonal matrix $\tB^n$
 by $D_p \tC(p)$. Thus $D_p\circ \tA_p^n \circ D_p^{-1}$ and $\tB^n$ are orthogonal matrices with
  the same eigenvalues and hence can be conjugate by an orthogonal matrix. It follows that $ \tA_p^n$ can be conjugate $\tB^n$ by a matrix in a bounded set.

\vskip.2cm
{\it First, we consider the case $f$ when has a fixed point $q=f(q)$.} 
Conjugating the cocycle $\A$ by $\tC(q)$ we can assume without loss of generality 
that the cocycles have the same value at $q$, that is $\tA_q=\tB_q=\tB$, and that $\tC(q)=\Id$.

Since $\tA$ is isometric, it is fiber bunched and hence has stable and unstable cocycle holonomies,
that is, 
$$H_{x,y}^s=H_{x,y}^{\tA,s}=\underset{n\to\infty}{\lim} (\tilde \A_y^n)^{-1}\circ \tilde \A_x^n\;\;
\text { exists for any $x\in X$ and any $y\in W^s(x)$, \,\,and}\;
$$
$$H_{x,y}^u=H_{x,y}^{\tA,u}=\underset{n\to\infty}{\lim} ((\tilde \A_y^{-n})^{-1}\circ \tilde \A_x^{-n})\;\;
\text { exists for any $x\in X$ and any $y\in W^u(x)$, }\;
$$
Above,  $W^s(x)$ and  $W^u(x)$ denote the stable and unstable leaves of a point $x\in X$.

The holonomies for the constant cocycle $\tB$ are trivial, that is, equal to the identity.
For a nontrivial H\"older continuous bundle $\TE$ one needs to consider appropriate 
H\"older continuous identifications of fibers at nearby points in the formulas above, and we refer to \cite{KS13} for more details. 
The existence and properties of holonomies (and their independence of the choice of identifications)  for this setting were established in \cite[Proposition 4.2]{KS13}. 
The holonomies $H^{s/u} _{x,\,y} : \TE_x \to \TE_y$ are
isomorphisms between the fibers of $\TE$ satisfying a natural equivariance property
\begin{equation} \label{H equiv}
\tA_x=H^{s/u} _{fy,\,fx} \circ \tA_y \circ H^{s/u} _{x,\,y}.
\end{equation}
They are also H\"older continuous along the leaves in the following sense
\begin{equation} \label{H holder}
\|H_{x,y}^{s/u} -\Id \|\le K\dist(x,y)^\beta \quad\text{for all $x \in X$ and all} \;y\in W^{s/u} _{loc}(x),
\end{equation}
where $\beta $ is a H\"older exponent of the cocycle and the constant $K$ depends on the size of the local leaves. This  follows from the estimates  \cite[Proposition 4.2(i)]{KS13}:
\begin{equation} \label{close to Id} 
\begin{aligned}
&\|(\tilde\A^n_y)^{-1} \circ  \tilde\A^n_x - \Id\,\| \leq K\,\dist (x,y)^{\beta}\,
\quad\text{for all $x\in X$, $y\in W^s_{loc}(x)$, $n\in \N$}, \\
&\|(\tilde\A^{-n}_y)^{-1} \circ \tilde\A^{-n}_x - \Id\,\| \leq K\,\dist (x,y)^{\beta}\,
\quad\text{for all $x\in X$, $y\in W^u_{loc}(x)$, $n\in \N$}.
\end{aligned}
\end{equation}

First we define conjugacies $C^s$ and $C^u$ on the stable and unstable leaves of the fixed point $q$. 
We set 
$$
\begin{aligned}
&C^s(q)=C^u(q)=\tC(q)=\Id,\\
 &C^s(x)=H^{s}_{q,\,x}\quad\text{for }x\in W^s(q),\\
 &C^u(x)= H^{u}_{q,\,x}\quad\text{for }x\in W^u(q).
 \end{aligned}
$$
Since $q=f(q)$ and $\tA_q=\tB$, using \eqref {H equiv} we obtain 
$$C^{s/u}(f x) \circ \tB_x \circ C^{s/u}(x)^{-1}=H^{s/u}_{q,\,fx} \circ \tB \circ (H^{s/u}_{q,\,x})^{-1}=
H^{s/u}_{q,\,fx} \circ \tA_q \circ H^{s/u}_{x,\,q}=\tA_x.
 $$
So  these are indeed conjugacies 
between $\tA_x$ and $\tB_x$ on the corresponding leaves.
\vskip.1cm

Now we show that if $x$ is a homoclinic point for $q$, 
that is, $x\in S :=W^s(q)\cap W^u(q)$, then 
\begin{equation} \label{Cs=Cu}
(C^s(x))^{-1}\circ C^u(x) = H^{s}_{x,q} \circ H^{u}_{q,x}= \Id,
\quad \text{that is, }\; C^s(x)=C^u(x).
\end{equation} 

By compactness of the orthogonal group we can find a sequence $n_k\to\infty$ such that $\tB^{n_k} \to \Id$.
Recall that we also have $\tA_q^{n_k}=\tB^{n_k}\to \Id$. For this sequence the holonomies
can be expressed as follows
$$
H_{x,q}^{s}=\lim _n((\tA_q^n)^{-1} \tA_x^n)=\lim _k((\tA_q^{n_k})^{-1} \tA_x^{n_k}) = \lim _k \tA_x^{n_k},
$$
and similarly as $q=f^{-n_k}(q)$
$$
H_{q,x}^{u}=\lim _k (\tA_x^{-n_k})^{-1} \tA_q^{-n_k} =\lim _k \tA_{f^{-n_k}(x)}^{n_k}( \tA_{f^{-n_k}(q)}^{n_k})^{-1}= \lim _k \tA_{f^{-n_k}(x)}^{n_k}.
$$
We conclude that
\begin{equation}\label{CsCu}
(C^s(x))^{-1}\circ C^u(x) = H^{s}_{x,q} \circ H^{u}_{q,x}
=\lim_k \tA_x^{n_k} \circ \tA_{f^{-n_k}(x)}^{n_k}= \lim_k \tA_{f^{-n_k}(x)}^{2n_k}.
\end{equation}

We  consider the orbit segment $\{f^{-n_k}x, \dots,x, \dots f^{n_k}(x)\}$.
Since both $f^{-n_k}x \to q$ and $f^{n_k}(x) \to q$ as $k \to \infty$, for large enough $k$
$$\delta_k= \dist (f^{-n_k}x,f^{n_k}(x))$$ 
is sufficiently small. Hence the Anosov Closing Lemma \cite[6.4.15-17]{KtH}) for $(X,f)$ applies and yields existence of a periodic point $p_k=f^{2n_k}(p_k)$ whose orbit 
is close to the orbit segment above, and in particular 
$$
\dist (f^{-n_k}x, p_k) \le K_1 \delta_k \quad \text{ and } \quad \dist ((f^{n_k}x), p_k) \le K_1 \delta_k. 
$$

Now we show  that $\tA_{f^{-n_k}(x)}^{2n_k}$ is close to $\tA_{p_k}^{2n_k}$. By the local product structure
there exits a unique point
$$y = W^s_{loc}(f^{-n_k}x)\cap W^u_{loc}(p_k).
$$ 
Then $ \dist (f^{-n_k}x,y) \le K_2 \delta_k$, and the orbit segment $\{y, \dots, f^{2n_k}(x)\}$ 
is close to both orbit segments $\{f^{-n_k}x,\dots, f^{n_k}(x)\}$ and $\{p_k, \dots, f^{2n_k}(p_k)=p_k\}$. 
It follows that  
$$
f^{2n_k}y = W^s_{loc}(f^{n_k}x)\cap W^u_{loc}(p_k),\; \text{ and so }\;\dist (f^{2n_k}y,p_k) \le K_2 \delta_k.
$$
Then applying \eqref{close to Id}
to $f^{-n_k}x$ and $y$,\, and to $f^{2n_k}y$ and $f^{2n_k}p_k=p_k$ we obtain
$$
\begin{aligned}
&\|  (\tA_{y}^{2n_k})^{-1} \tA_{f^{-n_k}(x)}^{2n_k}  -\Id \| \le K(K_2 \delta_k)^\beta \qquad \text{and}\\
&\| \tA_{p_k}^{2n_k} (\tA_{y}^{2n_k})^{-1}  -\Id \| =\| (\tA_{f^{2n_k}y}^{-2n_k})^{-1} \tA_{f^{2n_k}p_k}^{-2n_k}  -\Id \|\le K(K_2 \delta_k)^\beta.
\end{aligned}
$$
Since the cocycle $\tA$ is isometric, it follows that 
$$
\begin{aligned}
&\| \tA_{f^{-n_k}(x)}^{2n_k} - \tA_{p_k}^{2n_k} \| \le 
\| \tA_{f^{-n_k}(x)}^{2n_k} - \tA_{y}^{2n_k} \|+\| \tA_{y}^{2n_k} - \tA_{p_k}^{2n_k} \|=\\
&\| \tA_{y}^{2n_k}\big( (\tA_{y}^{2n_k})^{-1}\tA_{f^{-n_k}(x)}^{2n_k} - \Id\big) \|+\|\big( \Id - \tA_{p_k}^{2n_k} (\tA_{y}^{2n_k})^{-1} \big)\tA_{y}^{2n_k}\|\le 2K(K_2 \delta_k)^\beta.
\end{aligned}
$$
Hence $\| \tA_{f^{-n_k}(x)}^{2n_k} - \tA_{p_k}^{2n_k} \|\to 0$ as $k\to\infty.$ 

We recall that in both cases (i) and (ii) the set $\{\tilde C(p)\}$ is bounded.
Since $\tC(p_k)$ conjugates the periodic values of $\A$ and $\B$ at $p_k$, we have
$$
\begin{aligned}
&\| \tA_{p_k}^{2n_k} - \Id\|= \|\tC(p_k)^{-1} \tB^{2n_k}\tC(p_k) -\Id\| \le \\
& \|\tC(p_k)^{-1}\| \cdot \|\tC(p_k)\|\cdot \|\tB^{2n_k} -\Id\| \le K_3\|\tB^{2n_k} -\Id\|.
\end{aligned}
$$
Since $\tB^{2n_k} \to \Id$,\,  it follows that $\tA_{p_k}^{2n_k} \to \Id$. Thus using \eqref{CsCu} we conclude that
$$
(C^s(x))^{-1}\circ C^u(x)  = \underset{k}{\lim} \, \tA_{f^{-n_k}(x)}^{2n_k}= \Id, \quad\text{that is,}\quad C^s(x)=C^u(x).
$$

For each $x\in S =W^s(q)\cap W^u(q)$ we define
$$
\bar C(x) =C^s(x)=C^u(x).
$$
The function $\bar C$ is a conjugacy between $\tA$ and $\tB$ 
on the $f$-invariant homoclinic set  $S$, which is known to be dense in $\M$ \cite{Bo}. 
We will now show that the function  $\bar C$ is H\"older continuous 
on $S$. Then it  extends to a H\"older continuous function on $X$
which conjugates the cocycles. 

Let $x$ and $y$ be two sufficiently close points in $S$. 
We consider the point $z=W^u_{loc}(x)\cap W^s_{loc}(y)$, which is also in $S$.
By the definition of $\bar C=C^s$ 
$$
\bar C(z)= C^s(z)= H^{s}_{q,z} \qquad \bar C(y)= C^s(y)= H^{s}_{q,y},
$$
 and using properties of holonomies we obtain
\begin{equation}\label{CC^s}
 \bar C(z) \circ \bar C(y)^{-1} =  H^{s}_{q,z}  \circ (H^{s}_{q,y})^{-1} = 
 H^{s}_{q,z}  \circ H^{s}_{y,q} = H^s_{y,z}. 
\end{equation}
Similarly, using unstable holonomies, we obtain $C(x)\circ C(z)^{-1}= H^u_{z,x}$ and hence
\begin{equation}\label{CC^u}
\bar C(x)\circ \bar C(y)^{-1} = C(x)\circ C(z)^{-1}\circ C(z)\circ C(y)^{-1} = H^u_{z,x} \circ H^s_{y,z}. 
\end{equation}
Since 
$$
\begin{aligned}
& \|H_{z,x}^{u} -\Id \|\le K\dist(x,z)^\beta, \quad
\|H_{y,z}^{ u} -\Id \|\le K\dist(y,z)^\beta, \;\;\text{ and }\\
&\max \{ \dist(x,z),\dist(y,z)\} \le K_1 \dist(x,y),
\end{aligned}
$$ 
we conclude that
$$
\| \bar C(x)\circ \bar C(y)^{-1}  -\Id \|\le K'\dist(x,y)^\beta.
$$
We note that since the cocycle $\tA$ is isometric with respect to a H\"older continuous Riemannian metric,
its holonomies $H^s_{x,y}$ and $H^u_{x,y}$ are uniformly bounded in $x$ and $y$, and hence by definition $\|\bar C\|$ and $\|\bar  C^{-1}\|$ are bounded on $S$ by some constant $M$.
Hence 
\begin{equation} \label{d(CC)}
\begin{aligned}
 & d(\bar C(x),\bar C(y))=\|\bar C(x)-\bar C(y)\|+ \|\bar C(x)^{-1}-\bar C(y)^{-1}\|\le\\
 &\le  \|\bar C(x)\bar C(y)^{-1}-\Id \|\cdot \|\bar C(y)\|+
 \|\bar C(x)^{-1}\|\cdot \|\Id-\bar C(x)\bar C(y)^{-1}\| \le \\ 
 &\le 2MK' \,\dist(x,y)^\beta.
\end{aligned}
\end{equation}
We conclude that $\bar C$ is $\beta$-H\"older on $S$ and hence extends to a $\beta$-H\"older continuous function on $X$ which conjugates the cocycles.
 This completes the proof that $\tA$ and $\tB$ are $\beta$-H\"older cohomologous under the extra assumption that $f$ has a fixed point.
\vskip.2cm 

{\it Removing the fixed point assumption.}
 Now we deduce the result in full generality, 
 using the fact that $f$ has periodic points of all sufficiently large periods. 

Let $p_1$ be a periodic point of $f$ of some period $N$. Then $p_1$ is the fixed point for $f^N$ 
and we can apply the previous argument to the cocycles $\tA^N$ and $\tB^N$ over $f^N$.
Hence there exists a $\beta$-H\"older conjugacy $C_1$ between $\A^N$ and $\B^N$. 
We want to show that there exists a conjugacy between the original cocycles $\A$ and $ \B$ over $f$.
 
We define the {\em centralizer}\, of a $GL(d,\R)$ cocycle $\B$ as the set of self-conjugacies
 $$
    Z(\B)=\{ D:\M\to GL(d,\R)\; |\;\; \B_x=D(fx)\circ \B_x \circ D (x)^{-1} 
    \quad\text{for all }x\in \M \},
 $$
where all maps are considered  in the $\beta$-H\"older category. It is easy to see that $Z(\B)$ is a group with respect to pointwise multiplication. 
Let $\,\mathcal{C}(\B, \A)$ be the set of $\beta$-H\"older conjugacies between $\beta$-H\"older cocycles $\A$ and $\B$, and let $\,C_1 \in \mathcal{C}(\B, \A)$. Then $C_2 \in \mathcal{C}(\B, \A)$ \ if and only if $\;C_1C_2^{-1}\in Z(\B)$. 
Clearly $Z(\B)$ is a subgroup of $Z(\B^{k})$. It was shown in
\cite[Proposition 4.8]{S15}  that for any fiber bunched cocycle $\B$ there exists  
$M\ge 1$ such that 
$$\;Z(\B^{MT}) = Z(\B^M)\;\text{ for all  }\;T\ge 1.
$$
Applying this to $\tB^N$ we see that there exists $M$ such that
$$Z(\tB^{NM\cdot T}) = Z(\tB^{NM}) \;\text{ for all  }\;T\ge 1.$$ 
We pick a periodic point $p_2$ of a period $K>1$ relatively prime with $MN$. 
Using the argument with the fixed point for $f^K$ we obtain  a  $\beta$-H\"older conjugacy $C_2$ between the cocycles $\tA^K$  and $\tB^K$ over $f^K$.  
Since $C_1$ is a $\beta$-H\"older conjugacy between $\tA^N$ and $\tB^N$, 
it is also a conjugacy between $\tA^{NM}$ and $\tB^{NM}$. Similarly,   both $C_1$ and $C_2$ are $\beta$-H\"older conjugacies 
 between the cocycles $\tA^{NMK}$ and $\tB^{NMK}$ over $f^{NMK}$ and hence 
 $$C_1C_2^{-1} \in Z(\tB^{NMK})= Z(\tB^{NM}).
 $$ Now it follows that $\;C_2$ is also a conjugacy
between the cocycles $\tA^{NM}$ and $\tB^{NM}$.
Thus $C_2$ is a conjugacy for the cocycles over $f^{NM}$ and $f^K$,
where  $MN$ and $K$ are relatively prime.
Hence there exist integers $r$ and $s$ such that
$NMr+Ks=1$, and it is easy to see  that $C_2$ is also a conjugacy 
for the cocycles $\tA$ and $\tB$ over $f$.

This completes the proof of the theorem. $\QED$
\vskip.3cm



\section {Proof of Theorem \ref{relation theorem}} \label{rigidity proof}

First we establish that the conjugacy $C$ between the derivative cocycle $\A=Df$ 
and the constant cocycle $L$ is $\b$-H\"older for some $0<\beta<1$. 
Then we use this to show that a conjugacy $h$ between $f$ and $L$ is $C^{1+\b}$.
The bootstrap to $C^\infty$ regularity follows immediately from  \cite[Theorem 1.1]{KSW25}.

\subsection{H\"older continuity of the conjugacy between $Df$ and $L$.}$\;$\\
Let $\,0<\rho_1 <\dots <\rho_\ell\,$ be the distinct moduli of the eigenvalues of $L$, 
and let $\R^d= E^1\oplus\dots \oplus E^\ell$ be the corresponding Lyapunov splitting for $L$.

Suppose that the conjugacy $C$ is continuous. Then
\begin{equation} \label{splitting}
T\T^d =\E= \E^1 \oplus \dots \oplus \E^\ell, \;\text{ where }\E^i= C(E^i),
\end{equation}
is a continuous $\A$-invariant splitting. Moreover,  for each $i$, the cocycle $\A_i=\A|\E^i$ has the same  expansion/contraction rate as $L_i=L|E^i$. 
Thus the splitting  is dominated for $\A$, and hence it is $\b$-H\"older for some $0<\beta<1$.  Also, since $\|\A_i^n\|\cdot \|(\A_i^n)^{-1}\|$ grows at most polynomially as for $L_i$, it follows that each $\A_i$ is fiber bunched.
Then the conjugacy $C_i=C|E^i$ is $\b$-H\"older  by \cite[Theorem 2.1] {KSW23}
as a measurable conjugacy between a fiber bunched cocycle and a constant cocycle 
with one exponent. So we conclude that $C$ is $\b$-H\"older.

Now suppose that the conjugacy $C$ is measurable and $f$ is $C^1$ close to $L$.
Since $f$ is a $C^1$-small perturbation of $L$, we have a continuous $Df$-invariant 
splitting $\oplus\E^i$ close to $\oplus E^i$ with similar expansion/contraction rates. 
In particular, it is dominated for $\A$ and hence it is $\b$-H\"older for some $0<\beta<1$. 
Each $\A_i=\A|\E^i$ if fiber bunched as it is close to $L_i$. Also, we have $\E^i= C(E^i)$ almost
everywhere since any measurable conjugacy preserves the Lyapunov exponents of vectors, 
see e.g. \cite[Lemma 4.4] {KSW23}. 
Hence each $C_i=C|E^i$ is a measurable conjugacy between $\A_i$ and $L_i$, and it is $\b$-H\"older continuous as above by \cite[Theorem 2.1] {KSW23}. So we again conclude that $C$ is $\b$-H\"older.

\subsection{The main argument} Let $h$ be a topological conjugacy  between $L$ and $f$. 
Without loss of generality we can take  $h$ in the homotopy class of the identity. 
The stable and unstable foliations of $f$ are topological foliations with $C^{1+\beta}$ 
leaves, moreover the leaves vary continuously in $C^1$ topology and their tangent  bundles
$\E^s$ and $\E^u$ are $\b$-H\"older on $\T^d$. We say that such foliations have 
{\em uniformly $C^{1+\beta}$ leaves}. 

We will now prove that $h$  is a {\em uniformly $C^{1+\b}$ 
diffeomorphism} along $\w^s$. By this we mean that its restrictions $h|{\w^s}(x)$ to the stable leaves  are 
$C^{1+\b}$ diffeomorphisms that depend continuously on $x$ in $C^1$ topology and the derivative
$D_x(h|{\w^s}(x))$
is $\b$-H\"older  on $\T^d$. 
Similarly, $h$  is a uniformly $C^{1+\b}$ diffeomorphism along $\w^u$, and then $h$  is a $C^{1+\b}$ 
diffeomorphism of $\T^d$ by Journ\'e lemma \cite{J}. Let $1\le m <\ell$ be such that
$$
E^s= E^1 \oplus \dots \oplus E^m
$$
is the full stable sub-bundle for $L$. We denote by $W^i$ the invariant linear foliations of $L$ 
corresponding to Lyapunov sub-bundles $E^i$, and by $W^{j,m}$ the foliation corresponding to the sub-bundle $E^{j,m}=E^j \oplus \dots \oplus \E^m$ where $1 \le j \le m$. 

We use similar notations $\E^s= \E^1 \oplus \dots \oplus \E^m$ for the Lyapunov splitting of $f$.  By \cite[Lemma 4.3]{DG24} for each $1 \le j \le m$ the Lyapunov 
 subbundle $\E^j$ and  the weak stable subbundle 
$$\E^{j,m}=\E^j \oplus \dots \oplus \E^m$$ are tangent to $f$-invariant foliations with uniformly $C^{1+\b}$ leaves, 
denoted by $\w^i$ and $\w^{j,m}$ respectively. Moreover, each weak foliation $\w^{j,m}$ is mapped by the conjugacy to the corresponding linear foliation for $L$, that is, $h(\w^{j,m})=W^{j,m}$. 
In particular, we have $h(\w^{m})=W^m$. The main part of the proof is the following proposition.

\begin{proposition} \label{smooth along W}
If   $h$ maps $\w^j$ to $W^j$ for some $j$, then $h$  is a uniformly $C^{1+\b}$ diffeomorphism along $\w^j$.
\end{proposition}

We use this proposition in an inductive process showing that $h$  is a uniformly $C^{1+\b}$ diffeomorphism 
along the weak foliations $\w^{j,m}$
for $j=m, \dots ,1$. The base case is given by applying the proposition with  $j=m$,
and in the end we obtain smoothness along the full stable $\w^{1,m}=\w^s$. 

For the inductive step,
we assume that $h$  is a uniformly $C^{1+\b}$ diffeomorphism along $\w^{j+1,m}$.
Then we claim that the fast foliation $\w^{1,j}$, which always exists,
is mapped to the corresponding linear one, $h(\w^{1,j})=W^{1,j}$.
This is given by implication (4)$\implies$(1) of  \cite[Theorem1.1]{KS25}, and specifically follows
from \cite[Proposition 5.1]{KS25}. Its proof does not require regularity of $f$ higher than uniformly $C^{1+\b}$, 
and closeness of $f$ to $L$ is used only to obtain continuity of the splitting $\E^s= \E^1 \oplus \dots \oplus \E^m$. Since $L$ is weakly irreducible, the leaves of each Lyapunov foliation $W^i$ are dense in $\T^d$,
satisfying the assumption in  \cite{KS25}. This is the only place where weak irreducibility of $L$ is used
in the proof of Theorem \ref{relation theorem}. 

Since we always have $h(\w^{j,m})=W^{j,m}$, by intersecting we obtain that $h(\w^{j})=W^j$.
Now we apply Proposition \ref{smooth along W} to concude that $h$  is a uniformly $C^{1+\b}$ diffeomorphism along $\w^j$. Together with the assumed smoothness along $\w^{j+1,m}$, this yields by  Journ\'e lemma 
that $h$  is a uniformly $C^{1+\b}$ diffeomorphism along $\w^{j,m}$ and completes the inductive step.

To complete the proof of Theorem \ref{relation theorem} it remains to establish Proposition \ref{smooth along W}.


\subsection{Proof of Proposition \ref{smooth along W}}

We fix $j$ and write 
$$
\text{$\w$ for $\w^j$,\, $W$ for $W^j$,\, $\E$ for $\E^j$,\, and \,$E$ for $E^j$.}
$$
We will use nonstationatry linearizations of $f$ along $\w$ given by the following result.

\begin{lemma} \label{lin} \cite[Corollary 4.8]{K24} {\em(Non-stationary linearization)}. $\;$\\
Let $f$ be a $C^{1+\b}$, $0<\b<1$, diffeomorphism of a compact manifold $\M$,
and let $\w$ be an $f$-invariant topological foliation of $\M$ with uniformly $C^{1+\b}$ leaves. 
Suppose that for some continuous Riemannian metrics on $\E=T\w$ and 
constants  $\gamma<1<\hat \gamma \,$ with $\; \hat \gamma\gamma ^{1+\beta}< 1$,
$$
\hat  \gamma^{-1} <\|Df_x(v)\| < \gamma \;\;\text{ for any $x\in X$ and any unit vector  $v\in \E_x$}.
$$
Then there exists a unique family $\{ \h_x \} _{x\in \M}$ of $C^{1+\b}$ diffeomorphisms  
$\,\h_x: \w_x \to \E_x$ satisfying $\h_x(x)=0$ and $D_x \h_x =\Id \,$ such that  for each $x \in \M$,
 \begin{equation}\label{lin}
 Df|\E_{x} =\h_{f(x)} \circ f \circ \h_x ^{-1}:\,\E_{x} \to \E_{f{(x)}} . 
\end{equation}
The maps $\,\h_x |_{B^\w(x,R)}$ depend continuously on $x \in \M$ in $C^1$ topology and have $\b$-H\"older derivative  with uniformly bounded H\"older constant. 
 \vskip.2cm

\end{lemma}
Since the cocycle $\A=Df|\E$ is continuously conjugate to $L|E$, its quasiconformal distortion 
$\|\A^n_x\| \cdot \|(\A^n_x)^{-1}\|$ grows at most polynomially. Hence we can choose
a continuous metric so that the  distortion estimate $\gamma \hat \gamma$ is arbitrarily close to $1$,
and hence the bunching assumption $\; \hat \gamma\gamma ^{1+\beta}< 1$ is satisfied for any given $\beta>0$.

Since $h(\w)=W$, the foliations  $\w$ and $\w^u$  integrate to the joint foliation 
$\V = h^{-1}(V)$, where $V$ is the linear foliation corresponding to the subspace $ E \oplus E^u$.
We will use holonomies of the foliation $\w^u$ inside $\V$ between the leaves of $\w$ and will
denote them 
$$\H_{x,y}: \w(x) \to \w(y) \;\text{ for any }y\in \w^u(x). 
$$
The maps $\H_{x,y}$ are globally defined, continuous, and depend continuously in $C^0$ topology on $x$ and $y$ when restricted to a ball of fixed size. We will now show that the maps $\H_{x,y}$ are 
differentiable and their derivatives coincide with the corresponding cocycle holonomies. 
For this we apply the following proposition to $f^{-1}$, so that $\gamma$ and  
$ \hat \gamma$ in it are the same as  for $f$ in Lemma \ref{lin}.  The second condition $\, \hat \gamma\gamma ^{1+\beta}< 1$ in \eqref{bunch} 
is the same as in Lemma \ref{lin} and ensures existence of the nonstationatry linearization.
The first condition $\hat \gamma \gamma\, \nu^{\beta(1-\b)}< 1$ in \eqref{bunch} is stronger than  $\hat \gamma \gamma\, \nu^{\beta}< 1$,
 which is the s-bunching  of $\E$ that ensures existence of s-holonomies of the
 $\b$-H\"older cocycle $\A$. Both conditions are satisfied since $\gamma \hat \gamma$ 
can be chosen arbitrarily close to $1$.

\begin{lemma}\label{smooth hol} {\em (Smoothness of stable holonomies)}\\
Let $f$ be a $C^{1+\b}$ diffeomorphism of a compact manifold  $\M$.
Let $\w$ and $\w^s$ be transverse $f$-invariant topological foliations of $\M$ with uniformly $C^{1+\b}$ leaves 
which integrate to a joint topological foliation $\V$. Suppose that there
are  continuous Riemannian metrics on $\E=T\w$ and  $\E^s=T\w^s$ and constants
$\nu$,  $\gamma$,  and $\hat\gamma $  such that 
\begin{equation}\label{rates}
\begin{aligned}
&\|Df_x(v^s)\| < \nu < 1 < \gamma^{-1} <\|Df_x(v)\| < \hat\gamma\\
&\text{for any $x \in X$ and any unit vectors  $v^s\in \E^s_x$ and $v\in \E_x$.}
\end{aligned}
\end{equation}
Suppose that $\E$ is $\beta$-H\"older and $f$ satisfies the bunching assumptions
\begin{equation}\label{bunch}
\hat \gamma \gamma\, \nu^{\beta/(1+\b)}< 1 \quad \text {and} \quad \hat \gamma\gamma ^{1+\beta}< 1.
\end{equation}
Then for any  $x \in X$ and  $y\in \w^s(x)$, the local $\w^s$ holonomy 
$$\text{$\H_{x,y}: \w(x) \to \w(y)$ is differentiable and $D_x \H_{x,y} =H_{x,y},$}
$$
 where
$H_{x,y}=H_{x,y}^{s,\A}$ is the $\w^s$ holonomy of the cocycle $\A=Df|\E$.
 \end{lemma}
 
  If $f$ and the leaves of $\V$ were $C^{2+\b}$, this lemma could be obtained using 
 the $C^r$ Section Theorem. This result seems new even for $C^{1+\b}$ Anosov diffeomorphisms,
  yielding $C^{1+\b}$ regularity of the stable holonomies and the stable foliation if 
   $f$ is close to conformal on $\E^u$, or if $\E^u$ is one dimensional. 

\vskip.2cm 
 
\noindent {\it Proof of Lemma \ref{smooth hol}.}
The first bunching assumption in \eqref{bunch} yields that 
\begin{equation}\label{bu}
\hat\gamma\gamma\, \nu^\beta < \hat\gamma\gamma \,\nu^{\beta/(1+\b)}<\theta 
\quad \text{for some} \; \theta <1.
\end{equation}
 In particular, the cocycle $\A=Df|\E$ is s-bunched and hence it has $\beta$-H\"older s-holonomies
 $H_{x,y}=H_{x,y}^{s,\A}:\E_x \to \E_y$. The second bunching assumption in \eqref{bunch} 
 yields existence of nonstationary linearizations $\{\h_x\}_{x\in X}$ of $f$ along $\w$ by Lemma \ref{lin}.
\vskip.1cm  
We fix a small $0<\e_0<1$ so that  for any $x \in X$ and $y\in \w^s(x)$  with $d_{\w^s}(x,y)< \e_0/2$
the holonomy $\H_{x,y}$ of $\w^s$ is defined and satisfies $d_{\w^s}(z,\H_{x,y}(z))<\e_0$ on the ball 
of radius $\e_0$ in $\w(x)$ centered at $x$. We fix such $x$ and $y$ and 
 lift the holonomy $\H_{x,y}$ to $\E_x=T_{x}\w$ using the nonstationatry linearizations: 
 $$\bar \H_{x,y} =   \h_{f(x)} \circ \H_{x,y} \circ \h_x ^{-1}:\,\E_{x} \to \E_{f{(x)}}.
 $$
 Since $D_x \h_x =\Id=D_y \h_y  \,$ it suffices to show that $D_x \bar \H_{x,y} =H_{x,y}$. 
 We will prove that 
 $$\Delta (t)= \bar \H_{x,y} (t) - H_{x,y}(t)=o(t), \;\text{ where $t\in \E_x$.}
 $$
 We iterate forward and denote $x_n=f^n(x)$,  $y_n=f^n(y)$, and 
 $$\Delta_n = \bar \H_{x_n,y_n} - H_{x_n,y_n}.
 $$ 
 Using invariance properties of holonomies \eqref{H equiv} and linearizations \eqref{lin}, and denoting
$t_n=\A^n_x(t)$, we can write 
$$
 \Delta (t) =  (\A_y^n)^{-1} \circ \Delta_n \circ  \A_x^n (t)= (\A_y^n)^{-1} \Delta_n  (t_n). 
$$
Using \eqref{rates} we obtain
$$\| t_n\| \le \hat \gamma^n \|t \| \quad \text{ and } \quad \|\Delta (t)\| \le  \gamma^n \| \Delta_n  (t_n)  \|.$$
\vskip.05cm
Now we estimate $\Delta$.
\begin{equation}\label{Delta(t)}
\| \Delta (t)\| \,\le\, \gamma^n \| \Delta_n (t_n)\| \,\le\,  \gamma^n\| H_{x_n,y_n} (t_n) -  t_n\| + \gamma^n\| \bar \H_{x_n,y_n} (t_n) -  t_n\|.
\end{equation}
For the first term we use the H\"older property of cocycle holonomies \eqref{H holder}.
$$
\| H_{x_n,y_n} (t_n) -  t_n\| \,\le\, \| H_{x_n,y_n}  - \Id\| \,\|  t_n\| 
\,\le\,  Kd(x_n,y_n)^\beta \, \hat\gamma^n\, \|t \|
\,\le\, K  \e_0^\beta \nu^{n \beta }\, \hat \gamma^n \,\|t \|,
$$
and so using \eqref{bu} we obtain
\begin{equation}\label{Delta(t)2}
\gamma^n\| H_{x_n,y_n} (t_n) -  t_n\| \,\le\, \
K \gamma^n \hat \gamma^n\,\nu^{n \beta }\, \|t \| \,<\, K\theta^n \|t\|.
\end{equation}

\vskip.1cm

Now we estimate the second term in \eqref{Delta(t)}.  Denoting 
$$
z_n=\h_{x_n} ^{-1}(t_n) \in \w(x_n)\quad\text{and}\quad w_n =\H_{x_n,y_n}(z_n)\in \w(y_n)
$$ 
we have 
$$\bar \H_{x_n,y_n} (t_n) =\h_{y_n} \circ \H_{x_n,y_n} \circ \h_{x_n} ^{-1}(t_n) = \h_{y_n} (w_n).
$$ 
Using local coordinates we identify a small ball around 
$x_n$ with a ball in $\R^d$, where $x_n$ is identified with 0.
Then
\begin{equation} \label{H-t}
\| \bar \H_{x_n,y_n} (t_n) -  t_n\| \le \|t_n-z_n\| +\|z_n-w_n\|+\|w_n- \h_{y_n} (w_n)\|.
\end{equation}
For 
the first term we have
$$\|t_n-z_n\| =\|t_n-\h_{x_n} ^{-1}(t_n)\|\le M \| t_n\|^{1+\beta}\le M(\hat \gamma^n \|t \|)^{1+\beta}$$ 
since  $\h_x(x)=0$, $D_x \h_x =\Id=D_0\h_{x} ^{-1} $, and $C^{1+\beta}$ norms of $\h_{x_n}$ and $\h_{x_n}^{-1}$ on small balls are uniformly bounded. 
The middle term in \eqref{H-t} is estimated by contraction of $\w^s$ as 
$$\|z_n-w_n\|\le  \nu^{ n}\e_0.$$ 
The middle term bound decays while the first
 term bound grows with $n$, and we choose $n=n(t) \in \N$ to be the first
for which 
$$ \nu^{ n}\e_0 < M(\hat \gamma^n \|t \|)^{1+\beta}.
$$ 
For any sufficiently small $t$, this choice ensures that $n=n(t)$ is large and hence both that $\nu^{ n}\e_0$ and   $M(\hat \gamma^n \|t \|)^{1+\beta}$ are small. This also yields that $\hat \gamma^n \|t \|$ is small,  and so we can assume that 
$M(\hat \gamma^n \|t \|)^{1+\beta} \le \hat \gamma^n \|t \|.$
Then we have
$$
\begin{aligned}
\|y_n -w_n\| &\le \|y_n-0\|+ \|0-t_n\| +\|t_n-z_n\| +\|z_n-w_n\|\\
&\le \nu^{ n}\e_0 + \hat \gamma^n \|t \| +  
M(\hat \gamma^n \|t \|)^{1+\beta}+\nu^{ n}\e_0 \,\le\,  4\hat \gamma^n \|t \|.
\end{aligned}
$$
Now the third term in \eqref{H-t} can be estimated similarly to first one by 
 $$
 \|w_n- \h_{y_n} (w_n)\| \le M(\|w_n -y_n\|)^{1+\beta} \le M(4\hat \gamma^n \|t \|)^{1+\beta} \le 16M(\hat \gamma^n \|t \|)^{1+\beta}.
 $$
 We conclude that 
$$ \| \bar \H_{x_n,y_n} (t_n) -  t_n\| \le 18M(\hat \gamma^n \|t \|)^{1+\beta} .
$$

Using this and  \eqref{Delta(t)2} in   \eqref{Delta(t)} we obtain
$$
\| \Delta (t) \|\le 
 K\theta^n \|t\|  + \gamma^n 18M(\hat \gamma^n \|t \|)^{1+\beta},
$$
and hence
\begin{equation} \label{Delta(t)3}
\| \Delta (t) \| \,\|t \|^{-1}\le K\theta^n + 18M(\gamma\hat \gamma^{1+\b})^n \|t \|^{\beta}.  
\end{equation}
From the choice of $n=n(t)$ we have the following exponential estimate for $\|t\|$
$$ \nu^{ n}  \sim (\hat \gamma^n \|t \|)^{1+\beta} \iff  \|t \| \sim (\nu^{1/(1+\beta)} /\hat \gamma)^{n}.
$$
Hence the last term in \eqref{Delta(t)3}  can be bounded using \eqref{bu} as
$$ 18M(\gamma\hat \gamma^{1+\b})^n \|t \|^{\beta} \le M' (\gamma\hat \gamma^{1+\b})^n \cdot (\nu^{1/(1+\beta)} /\hat \gamma)^{\beta n}=
M' (\gamma\hat \gamma\, \nu^{\b/(1+\beta)})^n < M' \theta^n.
$$
We conclude that for $n=n(t)$,
$$\| \Delta (t) \| \,\|t \|^{-1}< (K+M') \theta^n.
$$
 Hence for any $\e>0$ 
there is $\delta>0$ such that $\|t\|<\delta$ implies that $n=n(t)$ is large enough and $(K+M') \theta^n<\e$.
This proves that $\| \Delta (t) \| =o(t)$ and completes the proof of the Lemma \ref{smooth hol}.
$\QED$

\vskip.3cm

We recall that $C$ is a $\beta$-H\"older conjugacy  between $\beta$ fiber bunched cocycles $\A=Df|\E$  and $L|E$. By \cite[Proposition 4.5]{S15}, $C$ 
 intertwines the  u-holonomy $H_{x,y}^{u,\A}:\E_x \to \E_y$ of $\A$ with the 
 trivial  holonomy $\Id :E_x \to E_y$ of the constant cocycle $L|E$,  that is, 
 \begin{equation} \label{inter}
 H_{x,y}^{u,\A}=C(y)\circ \Id\circ C(x)^{-1}, \quad\text{ and so }\quad H_{x,y}^{u,\A}\circ C(x) =C(y). 
\end{equation}
 
 We obtain a  $\beta$-H\"older Riemannian metric $g$ on the bundle $\E=T\w$ by pushing a constant Euclidean metric on $E$ by $C$. It is clear from the formula above that  the holonomy $H_{x,y}^{u,\A} $  is an isometry with respect to $g$.

We equip the leaves of $\w$ with the Riemannian metric $g$. Now Proposition \ref{smooth hol}, 
applied to $\w^u$ and $\w$ with $f^{-1}$, yields that 
the derivative $D_x \H_{x,y}$ of the holonomy $\H_{x,y}: \w(x) \to \w(y)$ of $\w^u$ coincides 
with $H_{x,y}^{u,\A}$, and thus it is  an isometry with respect to $g$. 
We conclude that $\H_{x,y}: (\w(x),g) \to (\w(y),g)$ is an isometry.

We fix an arbitrary $x\in \T^d$ and take any $z\in \w(x)$. Since the linear unstable leaf $W^{u}(x)$ is dense in $\T^d$, there exists a sequence of vectors $v_n\in E^{u}\subset\R^d$ such that $h(x)+v_n \in\T^d$ converges to $h(z)$. Denoting 
$y_n=h^{-1}(h(x)+v_n)$ we obtain a sequence of points $y_n\in \w^{u}(x)$ converging to $z$.
The linear holonomies $ \hat \H_{h(x),\,h(y_n)}$ of $W^u$ are given by translations by $v_n$, and  
hence they converge in $C^0$ to the translation $T_v$ in $W (h(x))$ by the vector $v=h(z)-h(x)$. Hence
the holonomies $\H_{x,y_n}$ also converge in $C^0$ norm to the corresponding homeomorphism 
$$
\t_{v}: \w (x) \to \w (x) \;\text{ given by }\;  \t_{v}=h^{-1} \circ T_{v} \circ h.
$$ 
Since the  holonomies $\H_{x,y_n}$ are isometries between $ \w (x)$ and $\w (y_n)$, the limit
$\t_{v}$ is an isometry of $( \w (x),g)$. Thus $h_x=h|\w(x)$ conjugates the action of $E =\R^k$ by translations of $W (h(x))$ 
with the corresponding continuous action of $\R^k$ by isometries $\t_v$ of  $ \w(x)$.
Denoting the group of  isometries of  $\w(x)$ by $G_x$ we obtain
an injective  continuous homomorphism 
 $$
 \eta_x : E  \to  G_x \quad\text{given by }\; 
 \eta_x (v)= \t_v =( h_x)^{-1} \circ T_v \circ  h_x.
 $$ 
Since the Riemannian metric $g$ is $\beta$-H\"older, the elements of $G_x$ are $C^{1+\beta}$ diffeomorphisms \cite{T}.  Classical results imply that $G_x$  is a finite dimensional Lie group. 

In our case this can be seen directly as follows. We claim that the $C^{1+\beta}$ nonstationary 
linearization $\h_x: (\w(x),g) \to (\E_x,g_x) $ is an isometry, giving a natural $C^{1+\beta}$ 
identification of $G_x$ with the Euclidean group of $\E_x$.
Indeed, it is easy to check that the family of maps
$\tilde H_{z,x}= D_z\h_x : \E_z \to \E_x=T_{\h(z)}\E_x$ is a  $\beta$-H\"older holonomy 
for  cocycle $\A=Df|\E$ along the foliation $\w$.  As in \eqref{inter} above, by \cite[Proposition 4.5]{S15} 
the conjugacy $C$ intertwines this (unique) holonomy for $\A$ with the trivial holonomy
 for $L|E$ , and hence $ D_z\h_x=\tilde H_{z,x}:(\E_z,g_z) \to (\E_x, g_x)$ is an isometry for each $z\in \w(x)$. 

Any continuous homomorphism between Lie groups is  a $C^\infty$  Lie group homomorphism, see
for example \cite[Corollary 3.50]{Ha}. Hence $\eta_x$ is  a $C^\infty$ diffeomorphism onto 
its image in $G_x$.  
Since $( h_x)^{-1}$  is  determined by  $\eta_x$ as 
 $$
( h_x)^{-1}(h(x)+v)=  \t_v (0)=\eta_x (v)(0),
 $$
 we obtain that $( h_x)^{-1}$ is a $C^{1+\beta}$ diffeomorphism between $W (h(x))$ and 
  $\w (x)$. 
Hence $h_x=h|\w(x)$ is also  a $C^{1+\beta}$ diffeomorphism. Since 
the homomorphisms $\eta_x$ depend continuously on $x$, so do their derivatives,
which yields that $h$ is uniformly $C^{1}$ along $\w$. Now the derivative 
$\tilde C (x)=D_x(h_x)$ is a continuous conjugacy between cocycles  $\A$ and $L|E$
and hence is $\b$-H\"older on $\T^d$ by \cite[Theorem 2.1] {KSW23}. This shows that $h$ is
uniformly $C^{1+\beta}$ along $\w$ and completes the proof of
Proposition \ref{smooth along W}. $\QED$


\end{document}